\documentclass[10pt]{amsart}

\usepackage[utf8]{inputenc}
\usepackage[T2A]{fontenc}
\usepackage[english]{babel}
\usepackage{amssymb,amsmath}
\usepackage{mathtools}
\usepackage{amsthm}
\usepackage{hyperref}
\usepackage{tikz-cd}
\usepackage{geometry}
\usepackage[shortlabels]{enumitem}
\newtheorem{theorem}{Theorem}[section]
\newtheorem{proposition}[theorem]{Proposition}
\newtheorem{lemma}[theorem]{Lemma}
\newtheorem{corollary}[theorem]{Corollary}

\theoremstyle{definition}
\newtheorem{definition}[theorem]{Definition}
\newtheorem{remark}[theorem]{Remark}

\newcommand{\Z}{\mathbb{Z}}
\newcommand{\Q}{\mathbb{Q}}
\newcommand{\Qell}{\mathbb{Q}_{\ell}}
\newcommand{\V}{\mathbb{V}}

\newcommand{\Zhat}{\widehat{\mathbb{Z}}}

\newcommand{\Oc}{\mathcal{O}}

\DeclareMathOperator{\Gal}{Gal}
\DeclareMathOperator{\Spec}{Spec}
\DeclareMathOperator{\GL}{GL}
\DeclareMathOperator{\Sec}{Sec}
\DeclareMathOperator{\qSec}{qSec}
\DeclareMathOperator{\Sel}{Sel}

\title{On reductions of Selmer sections}
\author{Wojciech Porowski}
\address{Research Institute for Mathematical Sciences, Kyoto University, Kyoto 606-8502, Japan}
\email{porowski@kurims.kyoto-u.ac.jp}  

\begin{document}

\begin{abstract}
We consider reductions of Selmer sections of the \'etale homotopy sequence of a hyperbolic curve over a number field.
We show that the conjugacy class of a noncuspidal Selmer section is uniquely determined by its reduction on a set of density one.
Moreover, we show that a noncuspidal Selmer section reduces to cusps only on a set of density zero, at least after a finite field extension.
\end{abstract}

\maketitle

\section{Introduction}\label{s:intro}
Let $X$ be a hyperbolic curve over a field $K$, consider the homotopy short exact sequence
\begin{equation}\label{s1:ses:etale}
1\to \Delta_X\to \Pi_X \to G_K\to 1,
\end{equation}
here $G_K$ is the absolute Galois group of $K$, $\Pi_X$ is the \'etale fundamental group of $X$ and $\Delta_X$ is the geometric \'etale fundamental group of $X$ (see~\cite[Exp.~IX, Thm.~6.1]{sga1}).
In this article we are interested in conjugacy classes of sections of the sequence~\eqref{s1:ses:etale}.
By functoriality of the \'etale fundamental group every $K$-rational point of $X$ determines a conjugacy class of sections called geometric sections.
Write $\overline{X}$ for the compactification of $X$, then every $K$-rational point of $\overline{X}\setminus X$ (i.e., a cusp) determines a packet of cuspidal sections (see~\cite{stix2012cusp} or Section~\ref{s:loc_cusp}).
The Grothendieck section conjecture, a major open problem in anabelian geometry, predicts that when $K$ is a number field every section of the sequence~\eqref{s1:ses:etale} is either geometric or cuspidal; see~\cite{stix2012rational} for a comprehensive introduction.

Assume that $K$ is a number field, write $\V(K)$ for the set of nonarchimedean valuations of $K$.
For a valuation $v\in \V(K)$ write $K_v$ for the completion of $K$ at $v$, $\Oc_v$ for the valuation ring and $\kappa(v)$ for the residue field.
We say that a section of the sequence~\eqref{s1:ses:etale} is \emph{Selmer} if for every valuation $v\in \V(K)$ the restriction of $s$ to a decomposition group $G_v\subset G_K$ of $v$ is either geometric or cuspidal.
Thus, every Selmer section $s$ determines a collection of $K_v$-rational points of $\overline{X}$
\[
s_v\in \overline{X}(K_v),\; \textrm{for}\; v\in \V(K).
\]
After choosing a model $\overline{\mathcal{X}}$ of $\overline{X}$ over the ring of integers $\Oc_K$ of $K$ we may consider the reductions $\bar{s}_v$ of points $s_v$, i.e., the images of $s_v$ through the map
\[
\overline{X}(K_v) = \overline{\mathcal{X}}(\Oc_v) \to \overline{\mathcal{X}}(\kappa(v)).
\]
The collection of points $(\bar{s}_v)_v$ is called the reduction of the Selmer section $s$; note that modulo a finite subset it may be regarded as being independent of the chosen model. 

Our first result, which is a strengthening of~\cite[Thm 1.1]{porowski2024} in the Selmer case, says that the reduction of a noncuspidal Selmer section on a large set of valuations uniquely determines the conjugacy class of the section; see Corollary~\ref{s5:cor:noncusp_selmer_weak_cover} where we prove a slightly stronger result.
\begin{theorem}\label{s1:thm:thm_1}
Let $s$ and $t$ be two noncuspidal Selmer sections of the sequence~\eqref{s1:ses:etale}, suppose that $s$ and $t$ have the same reduction on a set of valuations of density one.
Then $s$ and $t$ are conjugate.
\end{theorem}

Our second result (proved in Theorem~\ref{s5:thm:cusp_strict_reduction}) concerns the intersection of the reduction of a section with the divisor of cusps.
For a Selmer section $s$ of the sequence~\eqref{s1:ses:etale} write $\Omega_s$ for the set of all valuations $v\in \V(K)$ such that the reduction $\bar{s}_v$ coincides with the reduction of a cusp.
As before, modulo a finite set the subset $\Omega_s\subset \V(K)$ is independent of the chosen model.
We prove that in the noncuspidal case the set $\Omega_s$ has density zero after a finite field extension; this is a variant of \cite[Thm B]{stix2015birational} where birational Selmer sections were considered.

\begin{theorem}\label{s1:thm:thm_2}
Let $s$ be a noncuspidal Selmer section of the sequence~\eqref{s1:ses:etale}.
Then there exists a finite field extension $L/K$ such that the preimage of the set $\Omega_s$ under the natural restriction map
\[
\V(L)\twoheadrightarrow \V(K)
\]
has density zero.
\end{theorem}
Note that according to the section conjecture Theorem~\ref{s1:thm:thm_1} should hold when the set of valuations in its statement is only assumed to be infinite; similarly in Theorem~\ref{s1:thm:thm_2} the set $\Omega_s$ is conjectured to be finite.
However, bridging the gap between sets of density zero and infinite sets will require some new ideas beyond our methods as our approach is based on Chebotariev's density theorem thus seems to require some sort of positivity assumption.

We now sketch the structure of the present work, note that this article may be regarded as a natural continuation of the previous work~\cite{porowski2024} by the author. 
In Section~\ref{s:loc_cusp} we show that a section of the sequence~\eqref{s1:ses:etale} which is cuspidal on a set of density one must be globally cuspidal, this uses the main result of~\cite{porowski2024}.
In Section~\ref{s:weak_covers} we introduce weak covers of locally geometric sections and prove a variant of the finite covering property (cf.~\cite{porowski2024}) for them.
Section~\ref{s:ell-adic} contains a proof of a technical lemma on $\ell$-adic representations which is used later.
In Section~\ref{s:noncuspidal_selmer} we prove our main results by constructing a suitable $\ell$-adic representation of the \'etale fundamental group of $X$ and applying results obtained in previous sections.
Finally, in Section~\ref{s:fibres} we consider morphisms of curves and define fibres of morphisms over sections; we show that these fibres are finite with the expected bound on their size.

\section{Locally cuspidal sections}\label{s:loc_cusp}
Let $X$ be a smooth, geometrically connected variety over a field $K$, write $\overline{K}$ for an algebraic closure of $K$, $G_K= \Gal(\overline{K}/K)$ for the absolute fundamental group of $K$, $\Pi_X$ for the \'etale fundamental group $\pi^{et}_1(X)$ of $X$ and $\Delta_X$ for the geometric fundamental group $\pi^{et}_1(X_{\overline{K}})$, thus we have a short exact sequence of profinite groups
\begin{equation}\label{s2:seq:etale}
1\to \Delta_X\to \Pi_X\to G_K\to 1.
\end{equation}
When $X$ is a curve we say that $X$ is $\pi_1$\emph{-nontrivial} if the geometric fundamental group $\Delta_X$ is nontrivial.
Equivalently, $X$ is $\pi_1$-nontrivial if and only if $X$ is isomorphic neither to $\mathbb{P}^1$ nor to $\mathbb{A}^1$.
Denote by 
\[
\Sec(\Pi_X,G_K)
\]
the set of conjugacy classes of sections of the surjection $\Pi_X\twoheadrightarrow G_K$, we have a natural Kummer map 
\begin{equation}
X(K)\to \Sec(\Pi_X, G_K).
\end{equation}
A section of the sequence~\eqref{s2:seq:etale} whose conjugacy class lies in the image of the Kummer map is called \emph{geometric}.
When $X$ is either a $\pi_1$-nontrivial curve or a semiabelian variety and $K$ is either a number field or a $p$-adic local field then the Kummer map is injective.
In this situation a geometric section $s$ corresponds to a unique $K$-rational point $x\in X(K)$, we will then say that $s$ is geometric \emph{with support at}~$x$.

Assume now that $X$ is a curve.
Write $\overline{X}$ for the smooth compactification of $X$, it is a smooth proper curve over $K$.
We have a surjection $\Pi_{X}\twoheadrightarrow \Pi_{\overline{X}}$ induced by the open immersion $X\hookrightarrow \overline{X}$ and denote by $s^{cp}$ the composition of $s$ with this surjection.
Thus $s^{cp}$ is a section of the short exact sequence
\begin{equation}\label{s2:seq:etale_cp}
1\to\Delta_{\overline{X}}\to \Pi_{\overline{X}}\to G_K\to 1.
\end{equation}

We need to recall the notion of a cuspidal section; for more details see~\cite{stix2012cusp}.
Assume that $X$ is a hyperbolic curve and that $K$ has characteristic zero, denote the divisor of cusps by
\[
X^{csp} = \overline{X}\setminus X.
\]
Let $x\in X^{csp}(K)$ be a $K$-rational cusp and $D_x\subset \Pi_X$ be a decomposition group of $x$.
Write $\Oc_x$ for the local ring of $\overline{X}$ at $x$, $K_x$ for the field of fractions of $\Oc_x$ and $\widehat{K}_x$ for the completion of $K_x$.
Then the decomposition group $D_x$ might be identified with the absolute Galois group of $\widehat{K}_x$.
Denote by 
\[
I_x = D_x \cap \Delta_X
\]
the corresponding inertia group of the cusp $x$, we have an isomorphism $I_x\cong \Zhat(1)$ and a split short exact sequence of profinite groups
\begin{equation}\label{s2:seq:cusp}
1\to I_x\to D_x \to G_K\to 1.
\end{equation}
A section $s$ of the sequence~\eqref{s2:seq:etale} which comes from a section of the sequence~\eqref{s2:seq:cusp} for some $K$-rational cusp $x$ is called a \emph{cuspidal} section; the cusp $x$ is then uniquely determined by $s$ and we will say that $s$ is cuspidal \emph{with support at}~$x$.
The set of conjugacy classes of cuspidal sections at $x$ is a torsor under the group 
\[
H^1(G_K,I_x)\cong H^1(G_K,\Zhat(1))\cong \varprojlim_n K^\times/(K^\times)^n = \colon\widehat{K^\times}.
\]
Finally, let $X'$ be a smooth curve such that $X\subset X'\subset \overline{X}$ and $X = X'\setminus\{x\}$. We have a surjection $\Pi_{X}\twoheadrightarrow \Pi_{X'}$ and the image of a section $s$ of the sequence~\eqref{s2:seq:etale} determines a section $s'$ of the sequence
\[
1\to \Delta_{X'}\to \Pi_{X'}\to G_K\to 1.
\] 
When $s$ is cuspidal with support at $x$ then the section $s'$ is geometric with support at $x$.  

From now on we assume that $K$ is a number field, denote by $\V(K)$ the set of nonarchimedean valuations of $K$.
For a valuation $v\in \V(K)$ write $K_v$ for the completion of $K$ at $v$, $\Oc_v$ for the ring of integers of $K_v$ and $G_v\subset G_K$ for a decomposition group of $v$ (determined up to conjugacy).
Denoting 
\[
X_v = X\times_K K_v
\]
we also have a short exact sequence 
\begin{equation}\label{s2:seq:etale_v}
1\to \Delta_{X_v} \to \Pi_{X_v}\to G_v\to 1,
\end{equation}
which may be identified with the restriction of the sequence~\eqref{s2:seq:etale} to a decomposition group $G_v$.
Given a section $s$ of the sequence~\eqref{s2:seq:etale} we denote by $s_v$ its restriction to the sequence~\eqref{s2:seq:etale_v}; sections of the sequence~\eqref{s2:seq:etale_v} will be referred to as \emph{local} sections. 

\begin{definition}
Let $s$ be a section of the sequence~\eqref{s2:seq:etale} and $v\in \V(K)$ be a valuation.
We say that the section $s$ is \textit{cuspidal at} $v$ if the local section $s_v$ is cuspidal.
For a subset $\Omega\subset \V(K)$ we say that $s$ is \textit{cuspidal on} $\Omega$ if for every $v\in \Omega$ the section $s$ is cuspidal at $v$.
\end{definition}

Our first result shows that local cuspidality almost everywhere is equivalent to global cuspidality; this will be strengthened in later sections (see Corollary~\ref{s3:cor:cusp_red} and Theorem~\ref{s5:thm:cusp_strict_reduction}).
\begin{theorem}\label{s2:thm:loc_cusp}
Let $X$ be a hyperbolic curve over a number field $K$ and $\Omega\subset \V(K)$ be a subset of density one. Let $s$ be a section of the sequence~\eqref{s2:seq:etale} and suppose that the section $s$ is cuspidal on $\Omega$.
Then $s$ is a cuspidal section.
\end{theorem}
\begin{proof}
By passing to a neighbourhood of the section $s$ we may assume that $X$ has positive genus.
Consider the section $s^{cp}$ of the sequence~\eqref{s2:seq:etale_cp}.
Note that for every valuation $v\in \Omega$ the local section $s^{cp}_v$ is geometric, with support at a $K_v$-rational cusp of $X$.
Then, it follows from~\cite[Thm 4.6]{porowski2024} (or from~\cite[Thm 3.11]{stoll2007finite}) that there exists a $K$-rational cusp $x\in X^{csp}(K)$ such that $s^{cp}$ is geometric with support at $x$. 

Define a smooth curve $X'$ such that $X\subset X'\subset \overline{X}$ and $X = X'\setminus \{x\}$.
We have a commutative diagram with exact rows
\begin{equation}\label{s2:diag:partial_comp}
\begin{tikzcd}
1\arrow[r] & \Delta_X\arrow[r]\arrow[d, two heads] & \Pi_X\arrow[r]\arrow[d, two heads] & G_K\arrow[r]\arrow[d, equal] & 1\\
1\arrow[r] & \Delta_{X'}\arrow[r] & \Pi_{X'}\arrow[r] & G_K\arrow[r] & 1.
\end{tikzcd}
\end{equation}
Write $s'$ for the section of the second row of the diagram~\eqref{s2:diag:partial_comp} obtained by composing $s$ with the surjection $\Pi_{X}\twoheadrightarrow\Pi_{X'}$.
By the previous discussion and another application of~\cite[Thm 4.6]{porowski2024} we see that $s'$ is a geometric section with support at $x$. 

Let $t$ be a cuspidal section with support at $x$ of the first row of the diagram~\eqref{s2:diag:partial_comp}, write $t'$ for its composition with the surjection $\Pi_X\twoheadrightarrow\Pi_{X'}$.
Replacing $t$ by some of its conjugates we may assume that $t' = s'$. Let $D_x\subset \Pi_X$ be a decomposition group of the cusp $x$ such that the image of $t$ lies in $D_x$, we have a short exact sequence
\[
1\to I_x\to D_x\to G_K\to 1,
\]
where $I_x$ is the corresponding inertia group.
Consider the maximal cuspidally central subquotient $\Delta^{cc}_X$ of the surjection $\Delta_X \twoheadrightarrow \Delta_{X'}$, i.e., the maximal subquotient
\[
\Delta_X\twoheadrightarrow \Delta^{cc}_X \twoheadrightarrow \Delta_{X'}
\]
such that the image of the inertia group $I_x$ in $\Delta^{cc}_X$ is a central subgroup.
By the well-known description of the geometric \'etale fundamental group of a hyperbolic curve we have a short exact sequence
\begin{equation}\label{s2:seq:cc}
1\to I\to \Delta^{cc}_X\to \Delta_{X'}\to 1.
\end{equation}
Moreover $I\cong \Zhat(1)$ and the natural map $I_x\to I$ induced by the surjection $\Delta_X\twoheadrightarrow \Delta^{cc}_X$ is an isomorphism. 
Write $\Pi^{(cc)}_X$ for the quotient of $\Pi_X$ determined by the quotient $\Delta_X \twoheadrightarrow \Delta^{cc}_X$ hence we have a commutative diagram with exact rows
\begin{equation}\label{s2:diag:etale_to_cc}
\begin{tikzcd}
1\arrow[r] & \Delta_X\arrow[r]\arrow[d, two heads] & \Pi_X\arrow[r]\arrow[d, two heads] & G_K\arrow[r]\arrow[d, equal] & 1\\
1\arrow[r] & \Delta^{cc}_X\arrow[r] & \Pi^{(cc)}_X\arrow[r] & G_K\arrow[r] & 1.
\end{tikzcd}
\end{equation}
Write $s^{cc}$ and $t^{cc}$ for the composition of sections $s$ and $t$ with the surjection $\Pi_X\twoheadrightarrow \Pi^{(cc)}_X$, they are both sections of the lower row of the diagram~\eqref{s2:diag:etale_to_cc}.

Consider the action of $G_K$ on the sequence~\eqref{s2:seq:cc} via the section $t$.
Then the difference between sections $s^{cc}$ and $t^{cc}$ is a cocycle 
\[
G_K\ni \sigma \mapsto c_{\sigma}\in \Delta^{cc}_X.
\]
By construction sections $s'$ and $t'$ are equal thus we see that the cocycle $c_{\sigma}$ in fact takes values in the subgroup $I\subset \Delta^{cc}_X$. Hence $c$ determines a cohomology class
\[
c \in H^1(G_K,I).
\]
Now, recall that the set of cuspidal sections at $x$ is a torsor under the group $H^1(G_K,I_x)$ which is isomorphic to $H^1(G_K,I)$ via the isomorphism $I_x\to I$. Thus, replacing $t$ by an appropriate cuspidal section we may assume that the class $c$ is trivial hence sections $s^{cc}$ and $t^{cc}$ are conjugate.

We now claim that sections $s$ and $t$ are locally conjugate on $\Omega$, i.e., for every valuation $v\in \Omega$ the local sections $s_v$ and $t_v$ are conjugate.
Indeed, as $s_v$ is cuspidal with support at $x$, replacing $s_v$ by some of its conjugates we may assume that the image of $s_v$ is contained in $D_x$.
Then the difference between sections $s_v$ and $t_v$ is a cohomology class $c_v\in H^1(G_v,I_x)$. On the other hand, as sections $s^{cc}$ and $t^{cc}$ are conjugate, the class $c_v$ vanishes in $H^1(G_v,I)$ under the natural isomorphism map $I_x\to I$. Thus $c_v$ is trivial which proves the claim. 

Finally, it follows from~\cite[Thm 1.1]{porowski2024} that $s$ and $t$ are conjugate which finishes the proof since by construction $t$ is a cuspidal section.
\end{proof}

\section{Weak covers}\label{s:weak_covers}
This section contains a variant of a result obtained in~\cite[\S2]{porowski2024}. We keep the same notation as in the previous section. 

Let $A$ be a semiabelian variety over a number field $K$, write $T(A)$ for its Tate module. We have a commutative diagram 
\begin{equation}\label{s3:diag:kummer_for_abelian_variety}
\begin{tikzcd}
A(K)\arrow[r, hook] \arrow[d, hook] & H^1(G_K,T(A))\arrow[d, "loc_v"] \\
A(K_v)\arrow[r, hook] & H^1(G_v, T(A)),
\end{tikzcd}
\end{equation}
where $loc_v$ is the restriction map in group cohomology.
Let 
\[
c  \in H^1(G_K, T(A))
\]
be a cohomology class, we say that $c$ is \emph{geometric at} $v$ if the localized class $loc_v(c)$ comes from a point in $A(K_v)$ via the lower horizontal arrow in the diagram~\eqref{s3:diag:kummer_for_abelian_variety}.
Given a valuation $v\in \V(K)$ for which the class $c$ is geometric at $v$ we write $c_v$ for the unique $K_v$-rational point of $A$ mapping to $loc_v(c)$.

By spreading out, there exists an open subscheme $U\subset \Spec \Oc_K$ and a semiabelian scheme $\mathcal{A}\to U$ whose fibre over $K$ is isomorphic to $A$; we will say simply that $\mathcal{A}$ is a semiabelian model of $A$ and we fix one such model $\mathcal{A}$.
Let $v\in \V(K)$ be a valuation whose corresponding prime ideal lies in $U$, then we have an injection
\[
\mathcal{A}(\Oc_v)\hookrightarrow \mathcal{A}(K_v) = A(K_v).
\]
When $c$ is geometric at $v$ we say that the class $c$ in \emph{integral at} $v$ if the point $c_v$ arises from a (unique) $\Oc_v$-point of $\mathcal{A}$ via the above injection.  
Write $\kappa(v)$ for the residue field of $v$ and denote by $\bar{c}_v$ the image of $c_v$ along the reduction map 
\[
\mathcal{A}(\mathcal{O}_v) \to \mathcal{A}(\kappa(v)).
\]
The following definition we be convenient for our purposes.
\begin{definition}\label{s3:def:selmer_class}
Let $c$ be a cohomology class in $H^1(G_K,T(A))$. We say that $c$ is a \emph{Selmer} class if there exists a subset $\Omega\subset \V(K)$ of density one such that for every $v\in \Omega$ the class $c$ is integral at $v$.
\end{definition}
This definition is independent of the choice of a particular semiabelian model of $A$ as any two semiabelian models will be isomorphic over an open subscheme of $\Spec \Oc_K$.
It is easy to see that the collection of Selmer classes forms a subgroup of the group $H^1(G_K,T(A))$.
Moreover when $A$ is proper over $K$ (i.e., $A$ is an abelian variety) then our notion of a Selmer class is equivalent, modulo a set of density zero, to the usual definition of a Selmer class.


Let $c, c_1, \ldots, c_n$ be a collection of Selmer classes for some $n\ge 1$.
After choosing a semiabelian model $\mathcal{A}$ of $A$ we find a subset $\Omega\subset \V(K)$ of density one such that all classes $c,c_1\ldots, c_n$ are integral at every $v\in \Omega$. 
\begin{definition}\label{s2:def:weak_cover_coh}
We say that classes $(c_i)_i$ \textit{weakly cover} the class $c$ if there exists a subset 
\[
\Omega'\subset \Omega
\]
of density one such that for every $v\in \Omega'$ there exists an integer $1\le i\le n$ such that $\bar{c}_{i,v} = \bar{c}_v$.
\end{definition}
Similarly to the definition of a Selmer class the notion of a weak cover does not depend on the choice of a semiabelian model of $A$ and a subset of valuations $\Omega$.
With these definitions the main result of this section, which is a small modification of~\cite[Thm~2.9]{porowski2024}, may be stated as follows. 
\begin{theorem}\label{s3:thm:weak_cover}
Let $A$ be a semiabelian variety and 
\[
c,c_1,\ldots ,c_n\in H^1(G_K,T(A))
\]
be Selmer classes for some $n\ge 1$. Suppose that classes $(c_i)_i$ weakly cover the class $c$.
Then there exists $1\le i \le n$ such that $c_i = c$.
\end{theorem}
The proof of this theorem follows the same pattern as the proof of~\cite[Thm~2.9]{porowski2024} with a few minor modifications. Before giving the proof, we need a technical lemma concerning reductions of points on a semiabelian variety.
\begin{lemma}\label{s3:lem:reduction_mod_N}
Let $v\in \V(K)$ be a valuation lying over a prime $p\in \Z$ such that the absolute ramification index of $K_v$ is smaller than $p-1$. Consider a short exact sequence of commutative group schemes over $\Oc_v$
\[
0\to \mathcal{T}\to \mathcal{A}\to \mathcal{B}\to 0,
\]
where $\mathcal{T} = (\mathbb{G}^t_m)_{\Oc_v}$ for some integer $t\ge 0$ and $\mathcal{B}$ is an abelian scheme over $\Oc_v$.
Let $N\ge 1$ be a natural number which is prime to $p$. Then the reduction map $\mathcal{A}(\Oc_v)\to \mathcal{A}(\kappa(v))$ induces an isomorphism
\[
\mathcal{A}(\Oc_v)/N\mathcal{A}(\Oc_v) \cong \mathcal{A}(\kappa(v))/N\mathcal{A}(\kappa(v)).
\]
\end{lemma}
\begin{proof}
Since the torus $\mathcal{T}$ is split and $\Oc_v$ is a local ring we obtain from Hilbert's Theorem 90 a commutative diagram with exact rows
\[
\begin{tikzcd}
0\arrow[r] & \mathcal{T}(\Oc_v) \arrow[r]\arrow[d] & \mathcal{A}(\Oc_v)\arrow[r]\arrow[d] & \mathcal{B}(\Oc_v)\arrow[r]\arrow[d] & 0 \\
0\arrow[r] & \mathcal{T}(\kappa(v)) \arrow[r] & \mathcal{A}(\kappa(v))\arrow[r] & \mathcal{B}(\kappa(v))\arrow[r] & 0,
\end{tikzcd}
\]
here the vertical arrows are the reduction maps.
From the snake lemma we infer a commutative diagram with exact rows
\[
\begin{tikzcd}[cramped, sep=small]
\mathcal{B}(\Oc_v)[N] \arrow[r]\arrow[d] & \mathcal{T}(\Oc_v)/N\mathcal{T}(\Oc_v) \arrow[r]\arrow[d] & \mathcal{A}(\Oc_v)/N\mathcal{A}(\Oc_v) \arrow[r]\arrow[d] & \mathcal{B}(\Oc_v)/N\mathcal{B}(\Oc_v) \arrow[r]\arrow[d] & 1 \\
\mathcal{B}(\kappa(v))[N] \arrow[r] & \mathcal{T}(\kappa(v))/N\mathcal{T}(\kappa(v)) \arrow[r] & \mathcal{A}(\kappa(v))/N\mathcal{A}(\kappa(v)) \arrow[r] & \mathcal{B}(\kappa(v))/N\mathcal{B}(\kappa(v)) \arrow[r] & 1. 
\end{tikzcd}
\]
In the above diagram, the first vertical arrow is an isomorphism as the $N$-torsion group scheme $\mathcal{B}[N]$ is \'etale over $\Oc_v$.
By inspection, the second vertical arrow is also an isomorphism.
Finally, thanks to our assumptions the kernel of the reduction map
\[
\mathcal{B}(\Oc_v)\to \mathcal{B}(\kappa(v))
\]
is $N$-divisible (see~\cite[Appendix]{katz_1981}), hence the fourth vertical arrow in the above diagram is an isomorphism as well.
Thus it follows from the five lemma that the third vertical arrow is also an isomorphism.
\end{proof}

\begin{proof}[Proof of Theorem~\ref{s3:thm:weak_cover} (cf. the proof of Theorem 2.9 in~\cite{porowski2024})]
Considering the differences $c_i-c$ we may assume that $c=0$, our goal is to show that $c_i= 0$ for some $1\le i \le n$. Note that our assumptions are stable with respect to finite field extensions $L$ of $K$ and that the restriction maps
\[
H^1(G_K,T(A)) \hookrightarrow H^1(G_L,T(A))
\]
are injective.
Thus we are free to replace the base field $K$ by a finite field extension, in particular we will assume that the semiabelian variety $A$ is split.
Moreover, by discarding some classes if necessary we may also assume that for every class $c_i$ the set of valuations $v\in \V(K)$ such that $\bar{c}_{i,v} = 0$ has positive upper density.

Fix some semiabelian model $\mathcal{A}$ of $A$ over an open subscheme $U\subset \Spec \Oc_K$. By shrinking $U$ we may assume that there is a short exact sequence of commutative group schemes over $U$
\[
0\to \mathcal{T}\to \mathcal{A}\to \mathcal{B}\to 0,
\]
where $\mathcal{T} = (\mathbb{G}^t_m)_U$ for some integer $t\ge 0$ and $\mathcal{B}$ is an abelian scheme over $U$.
Write 
\[
\V(U)\subset \V(K)
\]
for the subset of all valuations $v\in \V(K)$ such that the prime ideal corresponding to $v$ lies in $U$.
By shrinking $U$ further we may assume that every valuation $v\in \V(U)$ has absolute ramification degree smaller than $p_v -1$, where $p_v\in \Z$ is the residue characteristic of $v$.
Let 
\[
\Omega\subset \V(U)
\]
be a set of valuations with density one such that all classes $(c_i)_i$ are integral at every $v\in \Omega$ (with respect to the chosen model $\mathcal{A}\to U$).
Denote $M = T(A)$ and for any integer $N\ge 1$ write $M_N= M/NM$. 

First we claim that some class $c_i$ is torsion.
Suppose otherwise, then by the same argument as in the proof of~\cite[Thm 2.9]{porowski2024} there exists an integer $N\ge 1$ and a set of valuations $S_0\subset \Omega$ of positive upper density such that for every $v\in S_0$ and every $1\le i \le n$ the image of $c_i$ through the map
\[
H^1(G_K, M) \to H^1(G_K,M_N)
\] 
is nontrivial.
After removing finitely many valuations from $S_0$ we may assume that all valuations in $S_0$ are prime to $N$. 
Let $v\in S_0$ be a valuation, consider the following commutative diagram 
\[
\begin{tikzcd}
H^1(G_K,M)\arrow[r, "loc_v"] & H^1(G_v,M)\arrow[r] & H^1(G_v,M_N)\\
& A(K_v)\arrow[u, hook]\arrow[r, two heads] & A(K_v)/NA(K_v)\arrow[u, hook].
\end{tikzcd}
\]
Since classes $(c_i)_i$ are geometric at $v$ it follows from the definition of $S_0$ that the images of $c_{i,v}$ in
\[
A(K_v)/NA(K_v)
\]
are nontrivial.
On the other hand we have a sequence of maps with the arrow on the right being an isomorphism by Lemma~\ref{s3:lem:reduction_mod_N}
\[
\begin{tikzcd}
A(K_v)/NA(K_v) & \mathcal{A}(\Oc_v)/N\mathcal{A}(\Oc_v)\arrow[l] \arrow[r, "\simeq"] & \mathcal{A}(\kappa(v))/N\mathcal{A}(\kappa(v)).
\end{tikzcd}
\]
Therefore, as classes $(c_i)_i$ are in fact integral at $v$, it follows that their images in
\[
\mathcal{A}(\kappa(v))/N\mathcal{A}(\kappa(v))
\]
are also nontrivial.
In particular for every $v\in S_0$ and every $1\le i\le n$ the reductions
\[
\bar{c}_{i,v}\in \mathcal{A}(\kappa(v))
\]
are nontrivial. Since $S_0$ has positive upper density this is a contradiction with the assumption that classes $(c_i)_i$ weakly cover the zero class.
Hence some class $c_i$ must be torsion.

To finish the proof it will be enough to show that a torsion class $c$ which is integral at every $v\in \Omega$ and such that $\bar{c}_v = 0$ for infinitely many $v\in \Omega$ must be trivial.
Indeed, if $dc = 0$ for some $d \ge 1$ then pick a valuation $v\in \Omega$ such that $d$ is prime to the characteristic of $\kappa(v)$ and $\bar{c}_v = 0$.
Since the $d$-torsion group scheme $\mathcal{A}[d]$ is \'etale over $\Oc_v$ we see that the reduction map is an injection on $d$-torsion points
\[
\mathcal{A}(\Oc_v)[d]\hookrightarrow \mathcal{A}(\kappa(v)).
\]
This implies that $c_v = 0$. Then, by the same argument as in the proof of \emph{loc.~cit.} we deduce that $c=0$.   
\end{proof}

We now return to considering fundamental groups of curves over a number field.
Let $X$ be a $\pi_1$-nontrivial curve over $K$, as in Section~\ref{s:loc_cusp} we are interested in sections of the short exact sequence
\begin{equation}\label{s3:seq:etale}
1\to \Delta_X\to \Pi_X\to G_K\to 1.
\end{equation}
For every valuation $v\in \V(K)$ we have a commutative diagram
\begin{equation}\label{s3:diag:kummer_maps}
\begin{tikzcd}
X(K)\arrow[r, hook] \arrow[d, hook] & \Sec(\Pi_X,G_K)\arrow[d]\\
X(K_v)\arrow[r, hook]  & \Sec(\Pi_{X_v},G_v),
\end{tikzcd}
\end{equation}
where the horizontal arrows are given by injective Kummer maps and the right horizontal arrow is the restriction map.
Write $\Delta_X\twoheadrightarrow \Delta^{ab}_X$ for the topological abelianization of $\Delta_X$ and consider the short exact sequence
\[
1\to \Delta^{ab}_X\to \Pi^{(ab)}_X\to G_K\to 1,
\]
here $\Pi^{(ab)}_X$ is the maximal geometrically abelian quotient of $\Pi_X$. 
For a section $s$ of the sequence~\eqref{s3:seq:etale} denote by $s^{ab}$ the composition of $s$ with the surjection $\Pi_X\twoheadrightarrow \Pi^{(ab)}_X$.
Then we have a commutative diagram
\begin{equation}\label{s3:diag:kummer_maps_ab}
\begin{tikzcd}
X(K)\arrow[r, hook] \arrow[d, hook] & \Sec(\Pi^{ab}_X,G_K)\arrow[d]\\
X(K_v)\arrow[r, hook]  & \Sec(\Pi^{ab}_{X_v},G_v),
\end{tikzcd}
\end{equation}
with injective horizontal arrows which is compatible with the diagram~\eqref{s3:diag:kummer_maps} via the map $s\mapsto s^{ab}$.

Let $s$ be a section of the sequence~\eqref{s3:seq:etale}, we introduce the following local conditions on $s$. 
\begin{definition}
Let $v\in \V(K)$ be a valuation, we say that $s$ is \textit{geometric at}~$v$ if the local section $s_v$ comes from a (unique) $K_v$-rational point of $X$.
For a subset $\Omega\subset \V(K)$ we say that $s$ is \textit{geometric on}~$\Omega$ if $s$ is geometric at every $v\in \Omega$. When $s$ is geometric on a set of density one we say that $s$ is \emph{almost locally geometric} and when $s$ is geometric on $\V(K)$ we say that $s$ is \emph{locally geometric}.
\end{definition}

To check that two almost locally geometric sections of $X$ are conjugate it is enough to prove it after passing to the Jacobian of the compactification $\overline{X}$.
\begin{lemma}\label{s3:lem:compact_ab}
Assume that $X$ has positive genus. Let $s$ and $t$ be two almost geometric sections of the sequence~\eqref{s3:seq:etale} such that $(s^{cp})^{ab}$ and $(t^{cp})^{ab}$ are conjugate.
Then $s$ and $t$ are conjugate.
\end{lemma}
\begin{proof}
Let $\Omega\subset\V(K)$ be a set of valuations of density one such that $s$ and $t$ are geometric on $\Omega$; fix a valuation $v\in \Omega$.
Then $(s^{cp}_v)^{ab}$ and $(t^{cp}_v)^{ab}$ are also conjugate.
Since the local sections $s^{cp}_v$ and $s^{cp}_v$ are geometric, it follows from the injectivity of the lower horizontal arrow in the diagram~\eqref{s3:diag:kummer_maps_ab} for $\overline{X}$ that sections $s^{cp}_v$ and $t^{cp}_v$ are conjugate.
Similarly, as $s_v$ and $t_v$ are geometric it follows that the local sections $s_v$ and $t_v$ are conjugate.
Therefore sections $s$ and $t$ are locally conjugate on $\Omega$ and we may apply~\cite[Thm 4.6]{porowski2024} to conclude.
\end{proof}

Fix a model $\overline{\mathcal{X}}$ of $\overline{X}$ over $\Oc_K$, namely a smooth proper morphism $\overline{\mathcal{X}}\to \Spec\Oc_K$ whose generic fibre is identified with the structure morphism $\overline{X}\to \Spec(K)$.
Let $\Omega\subset \V(K)$ be a nonempty set of valuations and $s$ be a section of the sequence \eqref{s3:seq:etale}, suppose that $s$ is geometric on $\Omega$.
Then for every valuation $v\in \Omega$ the local section $s_v$ determines a unique point in $X(K_v)$, which by abuse of notation we also denote by $s_v$.
Denote by $\bar{s}_v$ the image of $s_v$ along the reduction map
\[
X(K_v)\subset \overline{X}(K_v) = \overline{\mathcal{X}}(\Oc_v) \to \overline{\mathcal{X}}(\kappa(v)).
\]
Thus we obtain a collection of points
\[
(\bar{s}_v)_v,\; \textrm{for}\; v\in \Omega,
\]
called the \emph{reduction} of $s$ on $\Omega$.
This notion depends on a chosen model $\overline{\mathcal{X}}$, however any two models of $\overline{X}$ are isomorphic over a nonempty open subset of $\Spec \Oc_K$.
Therefore the reduction of a section on $\Omega$, modulo a finite set, may be regarded as being independent of this choice.
In our applications the set $\Omega$ will usually be a set of positive upper density, in particular infinite.

Consider now a collection $s,s_1,\ldots, s_n$ of almost locally geometric sections of the sequence~\eqref{s3:seq:etale} for some $n\ge 1$.
Let $\Omega\subset\V(K)$ be a set of valuations of density one such that all sections $s,s_1,\ldots,s_n$ are geometric on $\Omega$. 
We introduce the following analogue of Definition~\ref{s2:def:weak_cover_coh}.
\begin{definition}
We say that sections $(s_i)_i$ \textit{weakly cover} the section $s$ if there exists a subset 
\[
\Omega'\subset \Omega 
\]
of density one such that for every $v \in \Omega'$ there exists an integer $1\le i\le n$ such that $\bar{s}_{i,v} = \bar{s}_v$.
\end{definition}
Note that the notion of a weak cover is independent of the choice of a particular model of $X$ and the subset $\Omega$.
Also, this definition is a generalization of the notion of a cover of a section from~\cite[\S3]{porowski2024} for almost locally geometric sections.

With this definition, we can apply Theorem~\ref{s3:thm:weak_cover} to slightly improve the results of~\cite{porowski2024} as follows.
\begin{theorem}\label{s3:thm:weak_cover_sections}
Assume that $X$ is either hyperbolic or proper. Let 
\[
s, s_1,\ldots ,s_n
\]
be almost locally geometric sections of the sequence~\eqref{s3:seq:etale}, for some $n\ge 1$.
Suppose that sections $(s_i)_i$ weakly cover the section $s$.
Then there exists $1\le i\le n$ such that sections $s_i$ and $s$ are conjugate. 
\end{theorem}
\begin{proof}
Note that our assumptions are stable with respect to finite field extensions of $K$. Let $\Pi_{X'}\subset \Pi_X$ be a neighbourhood of the section $s$ corresponding to a finite \'etale cover $X'\to X$.
By considering splittings of sections $s_i$ in $X'$ (as in~\cite[\S 3]{porowski2024}) we see that after extending the base field $K$ it is enough to prove the statement for $X'$; note that this operation may increase the number $n$.
Thus we may assume that $X$ has positive genus.
Furthermore, thanks to Lemma~\ref{s3:lem:compact_ab}, we may replace $X$ by its compactification and assume that $X$ is a proper curve of positive genus.

Consider the short exact sequence
\[
1\to \Delta^{ab}_X\to \Pi^{(ab)}_X\to G_K\to 1.
\]
Recall that for every $1\le i\le n$ the difference between sections $s^{ab}_i$ and $s^{ab}$ is a cohomology class 
\[
c_i\in H^1(G_K,\Delta^{ab}_X)
\]
and $\Delta^{ab}_X$ is isomorphic to $T(A)$ where $A$ is the Jacobian variety of $X$.

By the previous discussion and the weak cover assumption we see that classes $(c_i)_i$ weakly cover the zero class.
Hence by Theorem~\ref{s3:thm:weak_cover} we see that there exists $1\le i\le n$ such that $s^{ab}_i$ and $s^{ab}$ are conjugate. Thus by Lemma~\ref{s3:lem:compact_ab} we conclude that $s_i$ and $s$ are conjugate. 
\end{proof}

\begin{remark}\label{s3:rem}
The conclusion of Theorem~\ref{s3:thm:weak_cover_sections} does not hold for $X=\mathbb{G}_m$.
Indeed, we claim that there exists an element $x \in \widehat{\Q^{\times}}$ such that for every prime number $p$ the image of $x$ in $\widehat{\Q^{\times}_p}$ lies in $\Q_p^{\times}$ and has positive $p$-adic valuation.
To construct such an element, choose a sequence $(n_q)_q$ of integers $n_q> 1$ indexed by all prime numbers $q$ such that for every prime number $p$ we have
\begin{enumerate}[(i)]
\item $p-1$ divides $n_q$ for almost all $q$,
\item $v_p(n_q) \to \infty$ as $q\to \infty$,
\end{enumerate}  
here $v_p$ is the $p$-adic valuation. Then define 
\[
x = \prod_q q^{n_q},
\]
and note that by the choice of the sequence $n_q$ the above product converges in every $\Z_p$ to an element of positive valuation.

Now, let $s$ be a geometric section of $(\mathbb{G}_m)_{\Q}$ corresponding to the point $1\in \mathbb{G}_m(\Q)$; consider $s$ as a trivialization of the $\widehat{\Q^{\times}}$-torsor of conjugacy classes of sections of the exact sequence
\[
1\to \Zhat(1) \to \Pi_{(\mathbb{G}_m)_{\Q}} \to  G_{\Q}\to 1.
\]
Pick any $y\in \Q^{\times}$ with $y \ne 1$ and consider two locally geometric sections $t_1$ and $t_2$ corresponding to elements $x$ and $xy$, respectively.
Then, regarding $\mathbb{P}^1_{\Z}$ as a model over $\Z$ of the compactification of $(\mathbb{G}_m)_{\Q}$ we see that for almost all prime numbers $p$ both sections $t_1$ and $t_2$ have the same reduction, equal to the point $0\in \mathbb{P}^1_{\mathbb{F}_p}$.
This means that the section $t_1$ weakly covers the section $t_2$, on the other hand they are clearly not conjugate as $y \ne 1$.
This does not contradict Theorem~\ref{s3:thm:weak_cover} since the cohomology class determined by $x$ is not a Selmer class. 
\end{remark}

We define another class of sections (see also~\cite[\S2.2.1]{stix2015birational}).
\begin{definition}
Assume that $X$ is either proper or hyperbolic and let $s$ be a section of the sequence~\eqref{s3:seq:etale}. We say that $s$ is a \textit{Selmer} section if for every $v\in \V(K)$ the section $s$ is either cuspidal at $v$ or geometric at $v$.
\end{definition}
Note that when $X$ is proper the notions of Selmer and locally geometric sections coincide.
We can extend the definition of the reduction of a section at a valuation $v\in \V(K)$ to the Selmer case, as follows.
Fix a model $\overline{\mathcal{X}}$ of $\overline{X}$ as before, for a Selmer section $s$ and a valuation $v\in \V(K)$ such that $s_v$ is cuspidal we define $\bar{s}_v$ as the image of the unique cusp determined by $s_v$ under the reduction map
\[
\overline{X}(K_v) = \overline{\mathcal{X}}(\Oc_v) \to \overline{\mathcal{X}}(\kappa(v)).
\] 
Let $s$ be a Selmer section of the sequence~\eqref{s3:seq:etale} and $\Omega\subset \V(K)$ be an infinite subset.
We say that the section $s$ \emph{reduces to cusps} on $\Omega$ if for all but finitely many $v\in\Omega$ the reduction $\bar{s}_v$ coincides with the reduction of a cusp of $X$.
As previously, this property does not depend on the choice of a model of $X$.
With this definitions we can generalize Theorem~\ref{s2:thm:loc_cusp} as follows.
\begin{corollary}\label{s3:cor:cusp_red}
Assume that $X$ is a hyperbolic curve.
Suppose that $s$ is a Selmer section of the sequence~\eqref{s3:seq:etale} which reduces to cusps on a set of valuations having density one.
Then $s$ is a cuspidal section.
\end{corollary}
\begin{proof}
By passing to a neighbourhood of $s$ we may assume that $X$ has positive genus.
Thanks to Theorem~\ref{s2:thm:loc_cusp} is it enough to show that for every valuation $v\in \V(K)$ the local section $s_v$ is cuspidal.
Extend the base field $K$ such that all cusps of $X$ become $K$-rational.
Write $s_1,\ldots,s_n$ for the geometric sections of the sequence
\[
1\to \Delta_{\overline{X}}\to \Pi_{\overline{X}}\to G_K\to 1
\]  
corresponding to the rational points of $\overline{X}$ determined by the cusps.
Then our assumptions imply that sections $(s_i)_i$ weakly cover the (locally geometric) section $s^{cp}$.
Thus we see from Theorem~\ref{s3:thm:weak_cover_sections} that for some $1\le i\le n$ sections $s^{cp}$ and $s_i$ are conjugate, hence for every $v\in \V(K)$ we have $s^{cp}_v= s_{i,v}$ as $K_v$-rational points of $X$. Therefore the local section $s_v$ is cuspidal.
\end{proof}

\section{\texorpdfstring{$\ell$}{Lg}-adic representations with eigenvalue one}\label{s:ell-adic}
In this section we prove a technical lemma which says that any $\ell$-adic representation of $G_K$ which admits a root of unity as an eigenvalue for many Frobenius elements in fact admits an open subgroup of $G_K$ whose elements enjoy the same property.
This will be used in the next section to strengthen our results from Sections~\ref{s:loc_cusp} and~\ref{s:weak_covers}.

For a finite extension $L/K$ of number fields, write 
\[
res_{L/K}\colon \V(L)\twoheadrightarrow \V(K)
\]
for the natural restriction map. We also define a subset
\[
Split(L/K)\subset \V(K)
\]
of all valuations of $K$ which are split completely in $L$.

\begin{definition}
Let $\Omega$ be a subset of $\V(K)$, we say that $\Omega$ has \textit{strongly positive} density if for every finite field extension $L/K$ the preimage $res^{-1}_{L/K}(\Omega)\subset \V(L)$ has positive upper density.
\end{definition}
Equivalently, a set $\Omega\subset \V(K)$ has strongly positive density if for every finite field extension $L/K$ the intersection $\Omega\cap Split(L/K)$ has positive upper density.  

\begin{remark}
We give a few examples of sets $\Omega$ which have strongly positive density.
\begin{enumerate}[(i)]
\item Any set of upper density one has strongly positive density.
\item For every tower of finite field extensions $K\subset L\subset M$ the set $res^{-1}_{M/K}(Split(L/K))$ has density one; in particular the set $\Omega = Split(L/K)$ has strongly positive density.
\item Let $\Omega\subset \V(K)$ be a section of the surjection $\V(K)\twoheadrightarrow \V(\Q)$.
Then for every finite field extension $L/K$ which is Galois over $\mathbb{Q}$ the preimage $res^{-1}_{L/K}(\Omega)$ has density $1/[K:\Q]$, in particular the set $\Omega$ has strongly positive density.
\end{enumerate}
\end{remark}
Let us also note that if $\Omega\subset \V(K)$ has strongly positive density and for some natural number $n\ge 1$ we are given subsets $\Omega_i\subset \V(K)$ for $1\le i\le n$ such that $\Omega\subset \bigcup_i \Omega_i$ then there exists $1\le i\le n$ such that $\Omega_i$ has strongly positive density.
\begin{remark}
Here we construct an example of a subset $\Omega\subset \V(K)$ with strongly positive density whose preimages $res^{-1}_{L/K}(\Omega)$ have arbitrarily small upper density for an increasing sequence of field extensions $L/K$.
In the following we denote the density by $\delta$ and the upper density by $\delta^{sup}$.
Pick any $\varepsilon > 0$ and consider a tower of finite Galois extensions of $K$
\[
K=K_0\subsetneq K_1\subsetneq K_2\subsetneq ...
\]
such that $\cup_i K_i = \overline{K}$.
Denote $S_i = Split(K_i/K)$, thus we have a descending sequence of sets
\[
\V(K)=S_0\supsetneq S_1\supsetneq S_2\supsetneq ...
\]
Note that by Chebotariev's density theorem we have $\delta(S_i) = 1/[K_i:K]$ thus $\delta (S_i\setminus S_{i+1}) > 0$ for every $i\ge 0$. Hence we may choose, for every $j\ge 0$, a set $\Omega_j\subset S_j\setminus S_{j+1}$ such that 
\[
0 < \delta^{sup}(\Omega_j) \le \frac{\varepsilon}{2^{j+1}[K_j:K]}.
\]
Fix $i\ge 0$, note that by construction $res^{-1}_{K_i/K}(\Omega_j)$ has density zero for all $j<i$. On the other hand, for $j \ge i$ we see that
\[
0 < \delta^{sup}(res^{-1}_{K_i/K}(\Omega_j))\le \frac{\varepsilon[K_i:K]}{2^{j+1}[K_j:K]}\le \frac{\varepsilon}{2^{j+1}}.
\]
Define $\Omega = \cup_j \Omega_j$ then for every $i\ge 0$ we have
\[
0 < \delta^{sup}(res^{-1}_{K_i/K}(\Omega)) \le \varepsilon (\frac{1}{2^{i+1}}+ \frac{1}{2^{i+2}}+ \ldots) = \frac{\varepsilon}{2^{i}},
\]
which goes to zero as $i$ goes to infinity.
\end{remark}

Let $\ell$ be a prime number, $V$ be a finite dimensional vector space over $\Qell$ and 
\[
\rho \colon G_K\to \GL(V)
\]
be a continuous representation. Denote by 
\[
\mathcal{G} = \rho(G_K)
\]
the image of $G_K$ in $\GL(V)$, thus $\mathcal{G}$ is an $\ell$-adic Lie group.
For each valuation $v\in \V(K)$ we have a surjective homomorphism 
\[
G_v\twoheadrightarrow G_{\kappa(v)}
\]
to the absolute Galois group of a finite field $\kappa(v)$; any element $F_v\in G_v$ whose image in $G_{\kappa(v)}$ is the Frobenius element will be called a \emph{Frobenius lift}. 

\begin{proposition}\label{s4:prop:eigenvalues_on_strict}
Let $\Omega\subset \V(K)$ be a set of valuations with strongly positive density. Suppose that for every $v\in \Omega$ there exists a Frobenius lift $F_v$ such that the operator $\rho(F_v)$ has an eigenvalue which is a root of unity.
Then there exists an open subgroup $U\subset G_K$ such that for every $g\in U$ the operator $\rho(g)$ has an eigenvalue which is a root of unity.
\end{proposition}
\begin{proof}
After replacing the representation $\rho$ by its semisimplification we may assume that $\rho$ is semisimple. In particular, by~\cite{khare_rajan_2001} the subset $S\subset \V(K)$ of valuations $v\in \V(K)$ such that $\rho$ is unramified at $v$ has density one.  

First we observe that there exists an integer $d\ge 1$ such that the roots of unity which appear as eigenvalues of $\rho(g)$ for $g\in G_K$ have order bounded by $d$. Indeed, since the characteristic polynomial of $\rho(g)$ is of degree $\dim V$ and with coefficients in $\Qell$, this follows from the fact that there are only finitely many extensions of $\Qell$ of bounded degree together with the finiteness of the group of roots of unity for local fields. We pick such an integer $d$.

By considering the resultant of the characteristic polynomial of the universal element of $\GL_V$ with the polynomial $x^d-1$ we see that there is a closed subscheme $Z\hookrightarrow \GL_V$ with the following property: for a linear map $M\in \GL_V(\Qell)$ we have $M\in Z(\Qell)$ if and only if $M$ has an eigenvalue which is a $d$th root of unity.
Let 
\[
\mathcal{Z} = \mathcal{G}\cap Z(\Qell)
\]
be the subset of $\mathcal{G}$ containing all elements which admit a $d$-th root of unity as an eigenvalue; thus $\mathcal{Z}$ is a closed analytic subset of $\mathcal{G}$.
Let $\Omega_0\subset S$ be the subset of valuations $v\in S$ such that $\rho(F_v) \in \mathcal{Z}$, by our assumption $\Omega_0$ has strongly positive density.

We now use a similar argument as in the proof of~\cite[\S6, Prop 13]{serre1981chebotarev}. Let $\mathcal{Z}_1$ be the interior of $\mathcal{Z}$ in $\mathcal{G}$ and $\mathcal{Z}_2 = \mathcal{Z}\setminus \mathcal{Z}_1$.
Since interiors of analytic subsets are closed we see that $\mathcal{Z}_1$ is an open and closed analytic subset of $\mathcal{G}$ and $\mathcal{Z}_2$ is a closed analytic subset with empty interior.  
Thus we have a decomposition 
\[
\Omega_0 = \Omega_1 \sqcup \Omega_2, 
\]
where for $i=1,2$ we write $\Omega_i$ for the subset of all $v\in\Omega_0$ such that $\rho(F_v)\in \mathcal{Z}_i$.
Choose a Haar measure on $\mathcal{G}$, as $\mathcal{Z}_2$ is an analytic subset with empty interior it follows that $\mathcal{Z}_2$ has measure zero.
Therefore from the Chebotariev's density theorem we see that $\Omega_2$ has density zero; in particular this implies that $\Omega_1$ and $\mathcal{Z}_1$ are nonempty.

As $\mathcal{Z}_1$ is a nonempty open and closed subset of $\mathcal{G}$, there exists an open subgroup $\mathcal{U}\subset \mathcal{G}$ such that $\mathcal{Z}_1$ is a disjoint sum of cosets of $\mathcal{U}$.
Since $\Omega_0$ has strongly positive density we see that the intersection $\mathcal{Z}_1\cap\mathcal{U}$ is nonempty, therefore $\mathcal{U}\subset \mathcal{Z}_1$.
Finally, by taking $U = \rho^{-1}(\mathcal{U})$ we obtain the desired subgroup of $G_K$. 
\end{proof}

\section{Noncuspidal Selmer sections}\label{s:noncuspidal_selmer}
In this section we will improve further on the results of Sections~\ref{s:loc_cusp} and~\ref{s:weak_covers} 	by relaxing the density assumption.

First we will recall a few facts concerning degenerations of \'etale covers of smooth curves.
Let $S$ be the spectrum of a discrete valuation ring $R$, write $K$ for the fraction field of $R$, $\kappa$ for the residue field of $R$ and $p\ge 0$ for the characteristic of $\kappa$.
Denote the generic point of $S$ by $\eta = \Spec K$ and the geometric generic point by $\bar{\eta} = \Spec \overline{K}$; similarly denote by $s = \Spec\kappa$ the closed point of $S$ and by $\bar{s} = \Spec \bar{\kappa}$ the geometric closed point.
Also, we fix a prime number $\ell$ which is different from the residue characteristic $p$.

Let $Y\to S$ be a pointed stable curve over $S$ with smooth generic fibre $Y_{\eta}$.
Thus $Y_{\bar{s}}$ is a pointed stable curve over $\bar{\kappa}$, write $\Gamma(Y_{\bar{s}})$ for the dual (semi) graph of $Y_{\bar{s}}$ (see~\cite[Appendix]{mochizuki_2004}).
There is a notion of the admissible fundamental group $\pi^{adm}_1(Y_{\bar{s}})$ of the pointed stable curve $Y_{\bar{s}}$ classifying tame admissible coverings of $Y_{\bar{s}}$, see~\cite{yang_2018} or~\cite{wewers_1999} for a definition and some basic properties.
Every tame admissible cover of a pointed stable curve is also a pointed stable curve.
The group $\pi^{adm}_1(Y_{\bar{s}})$ has also a combinatorial description as the fundamental group of a graph of groups whose underlying graph is isomorphic to $\Gamma(Y_{\bar{s}})$, see~\cite[Appendix]{mochizuki_2004}.     
There is a surjective specialization map
\[
\pi^{et}_1(Y_{\bar{\eta}}) \twoheadrightarrow \pi^{adm}_1(Y_{\bar{s}})
\]
which induces an isomorphism on maximal prime to $p$ quotients.
Finally, suppose we are given a finite \'etale map $Y'\to Y$ which is a Galois cover of degree prime to $p$.
Then it follows from Abhyankar's lemma that $Y'$ is also a pointed stable curve and the map $Y'\to Y$ is a morphism of pointed stable curves over $S$.

The next lemma will be used only in the case of residue characteristic zero; we state it in general for completeness.
\begin{lemma}\label{s5:lem:nontree_cover}
Suppose that the special fibre $Y_{\bar{s}}$ has at least two irreducible components. Then there exists a pointed stable curve $Y'$ over $S$ of genus $\ge 2$ and a finite morphism of pointed stable curves $Y'\to Y$ such that
\begin{enumerate}[(i)]
\item $Y'_{\bar{\eta}}\to Y_{\bar{\eta}}$ is a Galois \'etale cover whose degree is a power of $\ell$,
\item the dual graph $\Gamma(Y'_{\bar{s}})$ is not a tree.
\end{enumerate}
\end{lemma}
\begin{proof}
From the combinatorial description of the admissible fundamental group of $Y_{\bar{s}}$ we see that after passing to an appropriate admissible Galois cover of $Y_{\bar{s}}$ whose degree is a power of $\ell$ we may assume that the special fibre $Y_{\bar{s}}$ is sturdy, i.e., the normalization of every irreducible component of $Y_{\bar{s}}$ has genus at least two.
Let $v_1$ and $v_2$ be two distinct vertices of $\Gamma(Y_{\bar{s}})$ connected by an edge.
Then, after passing to an \'etale cover of $Y_{\bar{s}}$ with Galois group $\Z/\ell \Z$ which is connected over every irreducible component of $Y_{\bar{s}}$ we see that there will be at least $\ell$ edges connecting $v_1$ and $v_2$; in particular the graph $\Gamma(Y_{\bar{s}})$ will not be a tree.  
\end{proof}

As an application of the above lemma we can show that there are families of curves over hyperbolic curves whose degenerations at infinity are stable and with infinite monodromy; for our applications we need only the case of characteristic zero.

\begin{proposition}\label{s5:prop:family_of_curves}
Let $X$ be a hyperbolic curve over an algebraically closed field.
Then there exists a finite \'etale morphism $X'\to X$ and a family of smooth proper connected curves $Y\to X'$ of genus $\ge 2$ such that
\begin{enumerate}[(i)]
\item the family $Y\to X'$ extends to a family $\overline{Y}\to \overline{X'}$ of stable curves,
\item for every cusp $x'$ of $X'$ the dual graph of the fibre $\overline{Y}_{x'}$ is not a tree. 
\end{enumerate}
\end{proposition}
\begin{proof}
Consider the second configuration space of $X$
\[
Z = X\times X\setminus \Delta,
\] 
here $\Delta$ is the diagonal divisor.
The first projection $Z\to X$ makes it into a family of smooth curves over $X$ where the fibre over a closed point $x\in X$ is the hyperbolic curve $X\setminus \{x\}$.
Moreover, the family $Z\to X$ extends to a family of pointed stable curves $\overline{Z}\to \overline{X}$ and for every cusp $x\in X^{csp}$ the fibre $\overline{Z}_{x}$ has two irreducible components.
Write $\eta$ for the generic point of $X$ and let $x$ be a cusp of $X$. 
Let $R = \Oc_x$ be the local ring of $\overline{X}$ at $x$ and denote $S = \Spec R$.
Consider a pointed stable curve over $S$ obtained from the cartesian square
\[
\begin{tikzcd}
\overline{Z}_S \arrow[r]\arrow[d] & \overline{Z} \arrow[d] \\
S\arrow[r] & \overline{X}.
\end{tikzcd}
\]
Applying Lemma~\ref{s5:lem:nontree_cover} to $\overline{Z}_S\to S$ we find a normal open subgroup $U_x\subset\pi^{et}_1(Z_{\bar{\eta}})$, whose index is a power of $\ell$, corresponding to a pointed stable curve over $S$ with a finite map onto $\overline{Z}_S$ whose dual graph of the special fibre is not a tree.
Let $U$ be a normal open subgroup of $\pi^{et}_1(Z_{\bar{\eta}})$ obtained as the intersection
\[
U = \bigcap U_x
\]
over all cusps $x \in X^{csp}$; its index in $\pi^{et}_1(Z_{\bar{\eta}})$ is also a power of $\ell$.

Then it follows from~\cite[Prop 2.7]{stix_2005} that there exists a normal open subgroup $V\subset\pi^{et}_1(Z)$ whose preimage under the map
\[
\pi^{et}_1(Z_{\bar{\eta}})\to \pi^{et}_1(Z)
\]
coincides with $U$.
Write $Z''\to  Z$ for the \'etale cover of $Z$ determined by $V$ and $X'\to X$ for the \'etale cover of $X$ determined by the image of $V$ through the surjection $\pi^{et}_1(Z) \twoheadrightarrow \pi^{et}_1(X)$.
Thus we have a diagram with a cartesian square
\[
\begin{tikzcd}
Z'' \arrow[r] & Z' \arrow[r]\arrow[d] & Z \arrow[d] \\
              & X' \arrow[r]          & X,
\end{tikzcd}
\]
where the morphism $Z''\to Z'$ is a Galois \'etale cover whose degree is a power of $\ell$.

%


It follows from the discussion preceding Lemma~\ref{s5:lem:nontree_cover} that the family $Z''\to X'$ is a smooth pointed stable curve of genus $\ge 2$, write $Y\to X'$ for its contraction.
By the choice of the subgroup $U$ we see that the family $Y\to X'$ of smooth proper curves extends to a family $\overline{Y}\to \overline{X'}$ of stable curves of genus $\ge 2$ over the compactification $\overline{X'}$ and for every cusp $x'$ of $X'$ the fibre $\overline{Y}_{x'}$ is a stable curve whose dual graph is not a tree.
\end{proof}

We now come back to the situation preceding Lemma~\ref{s5:lem:nontree_cover}.
Consider the $\ell$-adic cohomology of the generic fibre. Denote 
\[
V = H^1_{et}(Y_{\bar{\eta}}, \Qell),
\]
then we have an induced Galois representation
\[
\rho \colon G_K\to \GL(V).
\]

\begin{lemma}\label{s5:lem:eigenvalue_of_nontree}
Suppose that $R$ is henselian and that the dual graph of the special fibre $Y_{\bar{s}}$ is not a tree. Then there exists a nonzero vector subspace
\[
W\subset V 
\] 
on which $G_K$ acts through a finite quotient. In particular, for any element $g\in G_K$ the linear operator $\rho(g)$ has an eigenvalue which is a root of unity.
\end{lemma}
\begin{proof}
This follows from~\cite[\S 9.7]{BLR_neron_models} and the theory of~\cite[Exp IX]{sga7_I}.
\end{proof}

After this preliminary discussion we return to the situation considered in previous sections, namely $X$ is a hyperbolic curve over a number field $K$ and we consider sections of the homotopy short exact sequence 
\begin{equation}\label{s5:ses:etale}
1\to \Delta_X\to \Pi_X\to G_K\to 1.
\end{equation}
In Corollary~\ref{s3:cor:cusp_red} we proved that a Selmer section of the sequence~\eqref{s5:ses:etale} which reduces to cusps on a set of density one must be cuspidal. 
We can now strengthen this result by relaxing the density assumption. 
\begin{theorem}\label{s5:thm:cusp_strict_reduction}
Let $s$ be a Selmer section of the sequence~\eqref{s5:ses:etale} and $\Omega\subset \V(K)$ be a set of valuations with strongly positive density.
Suppose that the section $s$ reduces to cusps on $\Omega$. Then $s$ is a cuspidal section. 
\end{theorem}

Before giving a proof of Theorem~\ref{s5:thm:cusp_strict_reduction} we need to recall the construction of a representation of $G_K$ from a family of smooth proper varieties over $X$ and a section of the sequence~\eqref{s5:ses:etale}.

Let $Y\to X$ be a family of smooth, proper curves over $X$. After choosing a prime number $\ell$ and a geometric point $\bar{x}\to X$ we have the induced representation
\[
\rho_{\bar{x}}\colon\pi^{et}_1(X,\bar{x}) \to \GL(H^1_{et}(Y_{\bar{x}},\Q_{\ell})).
\]
Different choices of a basepoint $\bar{x}$ will lead to isomorphic representations and in our applications we will be only interested in the isomorphism class of a representation.
Thus we will skip the geometric basepoint from the notation and write simply
\[
\rho \colon \Pi_X\to \GL(V),
\]
with $V$ a finite dimensional vector space over $\Q_{\ell}$.
Let $s$ be a section of the sequence~\eqref{s5:ses:etale}, by pulling back the representation $\rho$ along the map $s\colon G_K \hookrightarrow \Pi_X$ we obtain a representation
\[
\rho^s\colon G_K\to \GL(V).
\]
Let $v\in \V(K)$ be a valuation, we denote by
\[
\rho^s_v\colon G_v \to \GL(V),
\]  
the restriction of $\rho^s$ to a decomposition group $G_v\subset G_K$.
Note that when $s$ is geometric at $v$ then the representation $\rho^s_v$ is isomorphic to the natural action of $G_v$ on the vector space $H^1_{et}(Y_{\bar{x}},\Q_{\ell})$, where $\bar{x}$ is a geometric point over $s_v\in X(K_v)$.
Suppose now that $s$ is cuspidal at $v$, with support at a cusp $x\in X^{csp}(K_v)$.
Using the notation from Section~\ref{s:loc_cusp} we write $\eta = \Spec K_x$ and let $\bar{\eta}$ be a geometric point over $\eta$.
From the cartesian diagram
\[
\begin{tikzcd}
Y_{\bar{\eta}}\arrow[d]\arrow[r] & Y\arrow[d]\\
\bar{\eta}\arrow[r] & X
\end{tikzcd}
\]
we obtain an action
\begin{equation}\label{s5:eq:cuspidal_action}
D_x \cong \pi_1^{et}(\eta,\bar{\eta})\to \GL(H^1_{et}(Y_{\bar{\eta}}, \Q_{\ell}))
\end{equation}
which is isomorphic to the restriction of the representation $\rho$ to a decomposition subgroup $D_x\subset \Pi_X$.
Therefore, the representation $\rho^s_v$ is isomorphic to the pullback of the representation~\eqref{s5:eq:cuspidal_action} along a homomorphism $G_v\hookrightarrow D_x$ splitting the map $D_x\to G_K$ over $G_v$. 

\begin{proof}[Proof of Theorem~\ref{s5:thm:cusp_strict_reduction}]
Let $s$ be a Selmer section which reduces to cusps on a set of valuations $\Omega$ having strongly positive density. 
We are going to show that $s$ reduces to cusps on $\V(K)$, this will finish the proof thanks to Corollary~\ref{s3:cor:cusp_red}.
Note that we are free to replace the base field $K$ by a finite extension, in particular we may always assume that the cusps of $X$ are $K$-rational.

First, let $X'\to X$ be a finite \'etale cover corresponding to an open subgroup $\Pi_{X'}\subset \Pi_X$. After extending the base field $K$, the section $s$ splits into $n$ sections $s_1,\ldots, s_n$ of the sequence
\[
1\to \Delta_{X'}\to \Pi_{X'}\to G_K\to 1.
\] 
For every $1\le i\le n $ denote by $\Omega_i$ the set of valuations $v\in \V(K)$ such that $s_{i,v}$ is cuspidal.
Since $\Omega \subset \bigcup_i \Omega_i$, there exists an index $1\le i\le n$ such that $\Omega_i$ has strongly positive density.
Then it is enough to show that the section $s_i$ is cuspidal as this will imply the same for $s$. Hence to prove the theorem we are free to replace the curve $X$ by its finite \'etale cover.
Therefore, after applying Proposition~\ref{s5:prop:family_of_curves} and extending the base field $K$, we may assume that there exists a family of smooth proper curves $Y\to X$ over $K$ which extends to a family of stable curves $\overline{Y}\to \overline{X}$ such that for every cusp $x\in X^{csp}(K)$ the dual graph of the stable curve $Y_{\bar{x}}$ is not a tree.

Pick a prime number $\ell$ and consider the representation $\rho\colon\Pi_X\to \GL(V)$ associated to this family and its pullback
\[
\rho^s \colon G_K\to  \GL(V)
\]
via the section $s$. We claim that for all but finitely many valuations $v\in\Omega$ the linear operator $\rho^s(F_v)$ has an eigenvalue which is a root of unity.
This is clear if $s$ is cuspidal at $v$ with support at a cusp $x$, by applying Lemma~\ref{s5:lem:eigenvalue_of_nontree} to the stable curve over $S$ obtained from the cartesian diagram
\[
\begin{tikzcd}
\overline{Y}_S \arrow[r]\arrow[d] & \overline{Y} \arrow[d] \\
S = \Spec \widehat{\Oc_x}\arrow[r] & \overline{X}
\end{tikzcd}
\]
and using the previous description of the local representation $\rho^s_v$.
Hence we may concentrate on valuations $v$ at which $s$ is geometric.
By a spreading out argument we see that for almost all $v\in \Omega$ for which $s_v$ is geometric the smooth proper curve $Y_{s_v}$ over $K_v$ obtained from the cartesian diagram 
\[
\begin{tikzcd}
Y_{s_v} \arrow[r]\arrow[d] &  Y\arrow[d] \\
\Spec K_v\arrow[r, "s_v"] & X
\end{tikzcd}
\]
extends to a stable curve over $S = \Spec \Oc_v$ whose dual graph of the special fibre is not a tree.   
Hence we conclude again by Lemma~\ref{s5:lem:eigenvalue_of_nontree}.

Thus we may apply Proposition~\ref{s4:prop:eigenvalues_on_strict} to the representation $\rho^s$, therefore after extending the base field $K$ further by a finite extension we may assume that for every $g\in G_K$ the linear operator $\rho^s(g)$ has an eigenvalue which is a root of unity.
Let $v\in \V(K)$ be a valuation such that $s_v$ is geometric at $v$ with support at $x\in X(K_v)$.
Since the linear operator $\rho^s_v(F_v)$ has an eigenvalue which is a root of unity, it follows from the Riemann hypothesis for abelian varieties that the smooth proper curve $Y_{x}$ does not have good reduction.
Then again by a spreading out argument we conclude that the section $s$ reduces to cusps on $\V(K)$. 
\end{proof}
\begin{corollary}\label{s5:cor:noncusp_selmer}
Let $s$ be a noncuspidal Selmer section of the sequence~\eqref{s5:ses:etale}.
Then there exists a finite field extension $L$ of $K$ such that the restriction of $s$ to $G_L$ is an almost locally geometric section
\end{corollary}
\begin{proof}
This follows directly from Theorem~\ref{s5:thm:cusp_strict_reduction}.
\end{proof}

\begin{remark}
Write $A$ for the generalized Jacobian of $\overline{X}$ with respect to the divisor of cusps $X^{csp}$, it is a semiabelian variety over $K$.
After possibly extending $K$ and choosing a base point we get a map $X \to A$.
Let $s$ be a section of the sequence~\eqref{s5:ses:etale}, the map $X\to A$ will associate to the section $s$ a cohomology class $c\in H^1(G_K, T(A))$.
Suppose that the section $s$ is a noncuspidal Selmer section, then it follows from Theorem~\ref{s5:thm:cusp_strict_reduction} that after restricting to an open subgroup of $G_K$ the class $c$ is a Selmer class, as defined in Definition~\ref{s3:def:selmer_class}.
\end{remark}

We may also extend the definition of a weak cover of a section to the case of noncuspidal Selmer sections.
Namely, if $s,s_1,\ldots, s_n$ are noncuspidal Selmer sections for some $n\ge 1$ then we say that sections $(s_i)_i$ \emph{weakly cover} the section $s$ if there exists a set of valuations $\Omega\subset \V(K)$ of density one such that for every $v\in \Omega$ there exists $1\le i\le n$ such that $\bar{s}_{i,v} = \bar{s}_v$.
This allows us to prove a version of Theorem~\ref{s3:thm:weak_cover_sections} in the case of noncuspidal Selmer section as follows.  

\begin{corollary}\label{s5:cor:noncusp_selmer_weak_cover}
Let $s,s_1,\ldots s_n$ be noncuspidal Selmer sections of the sequence~\eqref{s5:ses:etale} for some $n\ge 1$.
Suppose that sections $(s_i)_i$ weakly cover the section $s$. Then for some $1\le i\le n$ sections $s_i$ and $s$ are conjugate.
\end{corollary}
\begin{proof}
Thanks to Corollary~\ref{s5:cor:noncusp_selmer} we may pass to a finite field extension of $K$ and assume that sections $s,s_1,\ldots,s_n$ are almost locally geometric.
Then the result follows from Theorem~\ref{s3:thm:weak_cover_sections}.
\end{proof}

\begin{remark}
In this remark we introduce another class of sections; this is not used in the rest of the paper.
Let $s$ be a locally geometric section of a hyperbolic curve $X$ over a number field $K$.
Fix a model $\overline{\mathcal{X}}$ of $\overline{X}$ over $\Oc_K$ and denote by $\Omega_s\subset \V(K)$ the subset of all valuations $v\in \V(K)$ such that $\bar{s}_v$ coincides with the reduction of a cusp of $X$.
Then the set $\Omega_s$, modulo a finite subset, is independent of the choice of the model $\overline{\mathcal{X}}$.
We say that the section $s$ is \emph{adelic} if the set $\Omega_s$ is finite.
From Theorem~\ref{s5:thm:cusp_strict_reduction} we see that the set $\Omega_s$ will have density zero after passing to a finite field extension; we may say that every locally geometric section becomes almost adelic after a finite field extension.

It is an important open problem to show that noncuspidal Selmer sections are locally geometric and adelic;
Theorem~\ref{s5:thm:cusp_strict_reduction} addresses this in a weak sense, modulo sets of density zero. 
Let us remark that a positive solution to this problem would imply in particular a positive solution to the birational version of the Grothendieck section conjecture.  
\end{remark}

\section{Fibres over sections}\label{s:fibres}
In the final section we consider fibres over almost locally geometric sections arising from a dominant morphism $X\to Y$ of curves and show that in most cases they are finite, with a bound on their size given by the degree of the morphism.

Let $A$ be a semiabelian variety over a number field $K$ and $\Omega\subset \V(K)$ be a subset of valuations.
Suppose that for every $v\in \Omega$ we are given a subset 
\[
C_v\subset A(K_v)
\]
of $K_v$-rational points of $A$. From this collection of subsets we define the set 
\[
C\subset H^1(G_K,T(A))
\]
of all cohomology classes $c$ such that for every $v\in \Omega$ the class $c$ is geometric at $v$ and moreover $c_v\in C_v$.

\begin{lemma}\label{s6:lem:list_ab}
Assume that $\Omega$ has density one. Let $n\ge 1$ be a natural number and suppose that for every $v\in \Omega$ the set $C_v$ is finite, of cardinality $\le n$. 
Then the set $C$ is finite, of cardinality $\le n$ 
\end{lemma}

\begin{proof}
Suppose that the conclusion does not hold, thus we can choose $n+1$ pairwise different classes
\[
c_0, c_1,\ldots, c_n
\]
belonging to $C$.
For $0\le i \neq j\le n$ let $c_{ij} = c_i - c_j$. By assumption, for every $v\in \Omega$ there exists $i\neq j$ such that the image of $c_{ij}$ under the restriction map
\[
H^1(G_K, T(A))\to H^1(G_v, T(A))
\] 
is trivial. Thus the classes $(c_{ij})_{ij}$ cover the zero class on $\Omega$, in the sense of~\cite{porowski2024}. Therefore by~\cite[Thm. 2.14]{porowski2024} it follows that $c_{ij}=0$ for some $i\ne j$ thus $c_i=c_j$, a contradiction.
\end{proof}

Let $X$ be a $\pi_1$-nontrivial curve over a number field $K$, recall that $\Sec(\Pi_X, G_K)$ denotes the set of conjugacy classes of sections of the homotopy exact sequence
\[
1\to \Delta_X \to \Pi_X \to G_K\to 1.
\]
For a finite field extension $L/K$ the restriction map
\[
\Sec(\Pi_X, G_K) \hookrightarrow \Sec(\Pi_{X_L}, G_L)
\]
is injective and we define the set of quasi-sections of $X$ as the colimit
\[
\qSec(X) = \varinjlim \Sec(\Pi_{X_L}, G_L),
\]
over all finite field extensions $L$ over $K$.
It is easy to extend the definition of various types of sections from Sections~\ref{s:loc_cusp} and~\ref{s:weak_covers} to the case of quasi-sections.
In particular we have the notions of (almost) locally geometric quasi-sections, as well as cuspidal and Selmer quasi-sections when $X$ is assumed to be hyperbolic.

Given a dominant morphism $f \colon X \to Y$ of $\pi_1$-nontrivial curves over $K$ we have an induced map
\begin{equation}\label{s6:eq:qsec_morph}
f \colon \qSec(X)\to \qSec(Y),
\end{equation}
denoted in the same way, defined by mapping a quasi-section of $X$ to its image via the homomorphism $\Pi_X\to \Pi_Y$.
This map is naturally compatible with the Kummer maps for $X$ and $Y$ as well as with localizing at valuations of $K$. Given a quasi-section $t\in \qSec(Y)$, the preimage $f^{-1}(t)$ of $t$ under the map~\eqref{s6:eq:qsec_morph} will be called the \emph{fibre} over the section $t$.
Finally, when $f$ is a finite \'etale morphism of degree $d\ge 1$ then for every quasi-section $t\in \qSec(Y)$ the fibre $f^{-1}(t)$ over $t$ is a finite set of cardinality $d$.

Now, we consider a situation analogous to the one considered in Lemma~\ref{s6:lem:list_ab} but in the case of curves.
Let $\Omega\subset \V(K)$ be a subset of valuations.
Suppose that for every $v\in \Omega$ we are given a subset 
\[
S_v\subset X(K_v)
\]
of $K_v$-rational points of $X$. From this collection of subsets we define the set 
\[
S\subset \Sec(\Pi_X,G_K)
\]
of all sections $s$ such that for every $v\in \Omega$ the section $s$ is geometric at $v$ and moreover $s_v\in S_v$.

\begin{lemma}\label{s6:lem:list}
Assume that $\Omega$ has density one and let $n\ge 1$ be a natural number. Suppose that for every $v\in \Omega$ the set $S_v$ is finite, of cardinality $\le n$. 
Then the set $S$ is finite, of cardinality $\le n$ 
\end{lemma}

\begin{proof}
Observe that when $\Delta_X$ is abelian then the statement follows directly from Lemma~\ref{s6:lem:list_ab} by translating between sections and cohomology classes. Thus we may assume that $X$ is a hyperbolic curve.
Moreover, we are also free to replace $K$ by a finite field extension.

Suppose that there are $n+1$ pairwise nonconjugate sections
\[
s_0,\ldots s_n
\]
belonging to the set $S$.
Consider a finite \'etale cover $f\colon X'\to X$ of degree $d\ge 1$, by extending the base field $K$ we may assume that for every $0\le i\le n$ all quasi-sections in the fibre $f^{-1}(s_i)$ are sections.
For $v\in \Omega$ denote 
\[
S'_v = f^{-1}(S_v)\cap X'(K_v),
\]
it is a finite set of cardinality $\le dn$.
Let 
\[
S'\subset \Sec(\Pi_{X'},G_K)
\]
be the set of all sections $s'$ such that for every $v\in \Omega$ the section $s'$ is geometric at $v$ and moreover $s'_v\in S'_v$.
Then we have an inclusion 
\[
\bigcup _{0\le i \le n} f^{-1}(s_i) \subset S',
\]
with the set on the left hand side having cardinality $d(n+1)$.
Therefore we see that to prove the statement for $X$ we may pass to a finite \'etale cover of $X$; in particular we assume that $X$ has positive genus.    

Clearly, any section contained in $S$ is an almost locally geometric section.
Thus we see that after applying Lemma~\ref{s3:lem:compact_ab} we are reduced to a statement about cohomology classes which was proved in Lemma~\ref{s6:lem:list_ab}.
\end{proof}
\begin{remark}
For a hyperbolic curve $X$ over a number field $K$ write
\[
\Sel(X,K)\subset \Sec(\Pi_X, G_K) 
\]
for the subset of Selmer sections.
Suppose that $X$ is proper and $K$ does not contain any CM subfield.
Then it is shown in~\cite{betts_stix_2025} that for every valuation $v\in \V(K)$ the restriction map
\[
\Sel(X,K) \to X(K_v)
\]
has finite image.
The previous lemma says in particular that a global bound (i.e., independent on $v$) on the cardinality of this image for all valuations $v$ from a set $\Omega$ of density one would imply the finiteness of the set $\Sel(X,K)$, with the same bound on its size.
\end{remark}

\begin{corollary}\label{s6:cor:almost_locally_geom}
Let $f\colon X \to Y$ be a dominant morphism of $\pi_1$-nontrivial curves over $K$, of degree $d\ge 1$. Let $t\in \qSec(Y)$ be an almost locally geometric quasi-section. Consider the set $S_0$ of all almost locally geometric quasi-sections $s\in \qSec(X)$ such that $f(s) = t$. Then $S_0$ is finite, of cardinality $\le d$.  
\end{corollary}
\begin{proof}
Suppose that there exists $n+1$ distinct quasi-sections
\[
s_0,\ldots, s_n
\]
in $S_0$, by passing to a finite field extension we may assume that $t$ and $(s_i)_i$ are in fact sections.
Let $\Omega\subset \V(K)$ be a set of valuations of density one such that sections $t$ and $(s_i)_i$ are geometric on $\Omega$.
For $v\in \Omega$ we define 
\[
S_v = f^{-1}(t_v)\cap X(K_v),
\]
it is a finite subset of $X(K_v)$ of cardinality $\le d$. 
Let
\[
S\subset \Sec(\Pi_X, G_K)
\]
be the set of all sections $s$ such that for every $v\in \Omega$ the section $s$ is geometric at $v$ and $s_v\in S_v$.
Then it is clear that $s_i\in S$ but by Lemma~\ref{s6:lem:list} the set $S$ has cardinality $\le n$, which is a contradiction.
\end{proof}

We now start describing fibres over Selmer sections along a dominant morphism of hyperbolic curves.

\begin{lemma}\label{s6:lem:preimage_of_selmer}
Let $f \colon X\to Y$ be a dominant morphism of hyperbolic curves over $K$. Let $t\in \qSec(Y)$ be a Selmer quasi-section. Suppose that the fibre $f^{-1}(t)$ is non-empty and let $s\in f^{-1}(t)$. Then the following hold:
\begin{enumerate}[(i)]
\item the quasi-section $s$ is Selmer,
\item if $t$ is cuspidal then $s$ is also cuspidal.
\end{enumerate}
\end{lemma}
\begin{proof}
The fact that $s$ is a Selmer section is proved in~\cite[Thm.~F]{mochizuki_tsujimura_2023}. Suppose that $t$ is cuspidal, as $s$ is a Selmer section we see that for every valuation $v\in \V(K)$ the local section $s_v$ is cuspidal.
Thus by Theorem~\ref{s2:thm:loc_cusp} the section $s$ is cuspidal.
\end{proof}

The previous lemma describes the fibre over a cuspidal section, in the next proposition we discuss the case of a noncuspidal Selmer section.

\begin{proposition}\label{s6:prop:fibre}
Let $f\colon X \to Y$ be a dominant morphism of hyperbolic curves over $K$, of degree $d\ge 1$.
Let $t\in \qSec(Y)$ be a noncuspidal Selmer quasi-section.
Then exactly one of the following holds:
\begin{enumerate}[(i)]
\item either there exists a cuspidal quasi-section in the fibre $f^{-1}(t)$ (thus $t$ is geometric),
\item or the fibre $f^{-1}(t)$ is finite, of cardinality $\le d$.
\end{enumerate} 
\end{proposition}
\begin{proof}
It is clear that both assertions cannot hold simultaneously, thus it is enough to show that the fibre $f^{-1}(s)$ contains at most $n$ noncuspidal quasi-sections.
Suppose that there exists $n+1$ distinct noncuspidal quasi-sections
\[
s_0,\ldots, s_n
\]
in the fibre $f^{-1}(t)$, by passing to a finite field extension we may assume that $(s_i)_i$ are in fact sections.
By applying Lemma~\ref{s6:lem:preimage_of_selmer} we see that sections $(s_i)_i$ are Selmer.
Thanks to Corollary~\ref{s5:cor:noncusp_selmer} we may extend the base field $K$ further and assume that sections $t$ and $(s_i)_i$ are almost locally geometric.
However, this contradicts the statement of Corollary~\ref{s6:cor:almost_locally_geom}.
\end{proof}

In certain situations we can show that the cardinality of the fibre (if nonempty) is the same as the degree of the morphism. The next proposition gives an example of such a statement, for simplicity we consider only the case of proper curves. 

\begin{proposition}
Let $f\colon X\to Y$ be a dominant morphism of proper hyperbolic curves, of degree $d\ge 1$, which is geometrically a Galois cover.
Let $t\in\qSec(Y)$ be a locally geometric quasi-section and suppose that the fibre $f^{-1}(t)$ is nonempty.
Then exactly one of the following holds:
\begin{enumerate}[(i)]
\item either $t$ is geometric, supported at a branch point of $f$,
\item or the fibre $f^{-1}(t)$ has cardinality $d$.
\end{enumerate}
\end{proposition}
\begin{proof}
After passing to a finite field extension of $K	$ we may assume that $Y\to X$ is a Galois cover with Galois group $G$, of size $d$.
Let $s$ be a quasi-section of $X$ mapping to $t$, by Lemma~\ref{s6:lem:preimage_of_selmer} we see that $s$ is a locally geometric section.
Extend the base field $K$ so that $s$ becomes a section.
Denote by $R\subset X$ the ramification divisor of $f$, it is a finite $K$-scheme.
Suppose first that for every $v\in \V(K)$ we have that $s_v\in R(K_v)$.
Then it follows from~\cite[Thm 3.11.]{stoll2007finite} that $s$ is a geometric section, supported at a ramification point of $f$.
Thus $t$ is geometric as well, with support at a branch point of $f$.

Otherwise, there exists a valuation $v\in \V(K)$ such that $s_v$ is not a ramification point of $f$.
Note that $G$ acts on the fibres $f^{-1}(t)$ and $f^{-1}(t_v)$ and the $G$-action is compatible with the restriction map
\[
f^{-1}(t)\to f^{-1}(t_v).
\]
Since $s_v$ is not a ramification point the orbit of $s_v$ has size $d$.
Thus we see that the orbit of $s$ also has size $d$, in particular the fibre $f^{-1}(t)$ has cardinality $\ge d$.
The reverse inequality follows from Proposition~\ref{s6:prop:fibre}. 
\end{proof}

\bibliographystyle{abbrv}
\bibliography{bibliography}

\begin{thebibliography}{10}

\bibitem{sga7_I}
{\em Groupes de monodromie en g\'eom\'etrie alg\'ebrique. {I}}, volume Vol. 288
  of {\em Lecture Notes in Mathematics}.
\newblock Springer-Verlag, Berlin-New York, 1972.
\newblock S\'eminaire de G\'eom\'etrie Alg\'ebrique du Bois-Marie 1967--1969
  (SGA 7 I), Dirig\'e{} par A. Grothendieck. Avec la collaboration de M.
  Raynaud et D. S. Rim.

\bibitem{sga1}
{\em Rev\^{e}tements \'{e}tales et groupe fondamental ({SGA} 1)}, volume~3 of
  {\em Documents Math\'{e}matiques (Paris) [Mathematical Documents (Paris)]}.
\newblock Soci\'{e}t\'{e} Math\'{e}matique de France, Paris, 2003.
\newblock S\'{e}minaire de g\'{e}om\'{e}trie alg\'{e}brique du Bois Marie
  1960--61. [Algebraic Geometry Seminar of Bois Marie 1960-61], Directed by A.
  Grothendieck, With two papers by M. Raynaud, Updated and annotated reprint of
  the 1971 original [Lecture Notes in Math., 224, Springer, Berlin; MR0354651
  (50 \#7129)].

\bibitem{betts_stix_2025}
L.~A. Betts and J.~Stix.
\newblock Galois sections and {$p$}-adic period mappings.
\newblock {\em Ann. of Math. (2)}, 201(1):79--166, 2025.

\bibitem{BLR_neron_models}
S.~Bosch, W.~L\"utkebohmert, and M.~Raynaud.
\newblock {\em N\'eron models}, volume~21 of {\em Ergebnisse der Mathematik und
  ihrer Grenzgebiete (3) [Results in Mathematics and Related Areas (3)]}.
\newblock Springer-Verlag, Berlin, 1990.

\bibitem{katz_1981}
N.~M. Katz.
\newblock Galois properties of torsion points on abelian varieties.
\newblock {\em Invent. Math.}, 62(3):481--502, 1981.

\bibitem{khare_rajan_2001}
C.~Khare and C.~S. Rajan.
\newblock The density of ramified primes in semisimple {$p$}-adic {G}alois
  representations.
\newblock {\em Internat. Math. Res. Notices}, (12):601--607, 2001.

\bibitem{mochizuki_2004}
S.~Mochizuki.
\newblock The absolute anabelian geometry of hyperbolic curves.
\newblock In {\em Proceedings of the {I}nternational {C}onference held at
  {T}okyo {M}etropolitan {U}niversity, {T}okyo, {S}eptember 25--29, 2001},
  volume~11 of {\em Developments in Mathematics}, pages 77--122. Kluwer
  Academic Publishers, Boston, MA, 2004.

\bibitem{mochizuki_tsujimura_2023}
S.~Mochizuki and S.~Tsujimura.
\newblock Resolution of nonsingularities, point-theoreticity, and
  metric-admissibility for p-adic hyperbolic curves.
\newblock preprint available on \url{http://hdl.handle.net/2433/284398}.

\bibitem{porowski2024}
W.~Porowski.
\newblock Locally conjugate {G}alois sections.
\newblock {\em J. Reine Angew. Math.}, 815:41--70, 2024.

\bibitem{serre1981chebotarev}
J.-P. Serre.
\newblock Quelques applications du th\'eor\`eme de densit\'e{} de {C}hebotarev.
\newblock {\em Inst. Hautes \'Etudes Sci. Publ. Math.}, (54):323--401, 1981.

\bibitem{stix_2005}
J.~Stix.
\newblock A monodromy criterion for extending curves.
\newblock {\em Int. Math. Res. Not.}, (29):1787--1802, 2005.

\bibitem{stix2012cusp}
J.~Stix.
\newblock On cuspidal sections of algebraic fundamental groups.
\newblock In {\em Galois-{T}eichm\"uller theory and arithmetic geometry},
  volume~63 of {\em Adv. Stud. Pure Math.}, pages 519--563. Math. Soc. Japan,
  Tokyo, 2012.

\bibitem{stix2012rational}
J.~Stix.
\newblock {\em Rational points and arithmetic of fundamental groups. {Evidence}
  for the section conjecture}, volume 2054 of {\em Lect. Notes Math.}
\newblock Berlin: Springer, 2013.

\bibitem{stix2015birational}
J.~Stix.
\newblock On the birational section conjecture with local conditions.
\newblock {\em Invent. Math.}, 199(1):239--265, 2015.

\bibitem{stoll2007finite}
M.~Stoll.
\newblock Finite descent obstructions and rational points on curves.
\newblock {\em Algebra Number Theory}, 1(4):349--391, 2007.

\bibitem{wewers_1999}
S.~Wewers.
\newblock Deformation of tame admissible covers of curves.
\newblock In {\em Aspects of {G}alois theory ({G}ainesville, {FL}, 1996)},
  volume 256 of {\em London Math. Soc. Lecture Note Ser.}, pages 239--282.
  Cambridge Univ. Press, Cambridge, 1999.

\bibitem{yang_2018}
Y.~Yang.
\newblock On the admissible fundamental groups of curves over algebraically
  closed fields of characteristic {$p>0$}.
\newblock {\em Publ. Res. Inst. Math. Sci.}, 54(3):649--678, 2018.

\end{thebibliography}

\end{document}